\documentclass[11pt,reqno]{amsart}

\usepackage{amssymb,amscd}

\usepackage{hyperref}
\hypersetup{colorlinks=false}

\usepackage{color}

\theoremstyle{plain}

\newtheorem{thm}{Theorem}[section]
\newtheorem{lem}[thm]{Lemma}
\newtheorem{cor}[thm]{Corollary}

\theoremstyle{definition}

\newtheorem{defn}[thm]{Definition}

\theoremstyle{remark}

\newcommand{\N}{{\mathbb N}}
\newcommand{\Z}{{\mathbb Z}}
\newcommand{\R}{{\mathbb R}}

\newcommand{\FF}{{\mathcal{F}}}
\newcommand{\GG}{{\mathcal{G}}}
\newcommand{\HH}{{\mathcal{H}}}

\newcommand{\fH}{{\mathfrak{H}}}
\newcommand{\fG}{{\mathfrak{G}}}

\newcommand{\bh}{\mathbf{h}}

\newcommand{\im}{\operatorname{im}}

\newcommand{\id}{\operatorname{id}}

\newcommand{\dom}{\operatorname{dom}}
\newcommand{\vol}{\operatorname{vol}}
\newcommand{\germ}{\operatorname{germ}}
\newcommand{\card}{\operatorname{card}}
\newcommand{\PSL}{\operatorname{PSL}}

\definecolor{darkgreen}{cmyk}{1,0,1,.2}
\definecolor{m}{rgb}{1,0.1,1}

\newdimen\theight
\def\TeXref#1{%
             \leavevmode\vadjust{\setbox0=\hbox{{\tt
                     \quad\quad  {\small \textrm #1}}}%
             \theight=\ht0
             \advance\theight by \lineskip
             \kern -\theight \vbox to
             \theight{\rightline{\rlap{\box0}}%
             \vss}%
             }}%

\title{Growth of some transversely homogeneous foliations}

\author[J.A. \'Alvarez L\'opez]{Jes\'us A. \'Alvarez L\'opez}
\address{Departamento de Xeometr\'{\i}a e Topolox\'{\i}a\\
         Facultade de Matem\'aticas\\
         Universidade de Santiago de Compostela\\
         15782 Santiago de Compostela\\ Spain}
\email{jesus.alvarez@usc.es}
\thanks{The first author is partially supported by MICINN, grant MTM2011-25656}

\author[R. Wolak]{Robert Wolak}
\address{Jagiellonian University\\
         Faculty of Mathematics and Computer Science\\
         Institute of Mathematics\\
         ul. prof. Stanis\l awa \L ojasiewicza 6\\ 
         30-348 Krak\'ow\\ 
         Poland}
\email{robert.wolak@im.uj.edu.pl}

\date{\today}

\subjclass{57R30}

\keywords{Transversely homogeneous foliation, pseudogroup, growth}

\begin{document}

\maketitle

\begin{abstract}
  For transversely homogeneous foliations on compact manifolds whose global holonomy group has connected closure, it is shown that either all holonomy covers of the leaves have polynomial growth with degree bounded by a common constant, or all holonomy covers of the leaves have exponential growth. This is an extension of a recent answer given by Breuillard and Gelander to a question of Carri\`ere. Examples of transversely projective foliations satisfying the above condition were constructed by Chihi and ben~Ramdane.
\end{abstract}

\section{Introduction}\label{s: intro}

Let $G$ be a Lie group, and $P\subset G$ a closed subgroup. The concept of (transversely homogeneous) $(G,G/P)$-foliation was introduced by Blumenthal \cite{Blumenthal1979}: it is a smooth foliation with the structure given by a defining cocycle with values in the $G$-manifold $G/P$. 
  
Assume that the $G$-action on $G/P$ is faithful, and $G/P$ is connected. Let $\FF$ be a $(G,G/P)$-foliation on a compact connected manifold $M$. Then Blumenthal proved that there is a homomorphism $\bh:\pi_1(M)\to G$, well-defined up to conjugation, whose image is denoted by $\Gamma$, such that the lift $\widetilde\FF$ to the cover $\widetilde M$ of $M$ associated to $\ker\bh$ is given by a $\Gamma$-equivariant submersion $D:\widetilde M\to G/P$, where $\Gamma$ acts on $\widetilde M$ via deck transformations \cite[Theorem~1]{Blumenthal1979}. It is said that $\Gamma$ is the ({\em global\/}) {\em holonomy group\/} of $\FF$. 

Moreover Blumenthal proved that, if $\pi_1(M)$ has non-exponential growth (respectively, polynomial growth of degree $d$), then all leaves of $\FF$ have non-exponential growth (respectively, polynomial growth of degree $d$) \cite[Theorem~1]{Blumenthal1979}. (Recall that the growth is one of the classical invariants of leaves of foliations on compact manifolds \cite{Plante1973,Plante1975}.)
  
When $\FF$ is a Lie $G$-foliation ($P=\{1\}$), Carri\`ere has shown that \cite{Carriere1988}:
  \begin{itemize}
    
    \item the leaves of $\FF$ are F\o lner if and only if $G$ is solvable; 
      
    \item the leaves of $\FF$ have polynomial growth if and only if $G$ is nilpotent; and, 
      
    \item  in the last case, the leaves of $\FF$ have polynomial growth with degree less than or equal to the degree of nilpotence of $G$. 
      
  \end{itemize}
Carri\`ere asked in \cite{Carriere1988} about the existence of a Lie $G$-foliation on a compact manifold whose leaves have neither polynomial nor exponential growth. This question was recently answered by Breuillard and Gelander \cite[Theorem~10.1]{BreuillardGelander2007}, obtaining the following dichotomy as a consequence of their study of a topological Tits alternative: 
  \begin{itemize}
  
    \item either all leaves of $\FF$ have polynomial growth with degree bounded by a common constant; 
    
    \item or all leaves of $\FF$ have exponential growth.
  
  \end{itemize}
Indeed, Carri\`ere, and Breuillard and Gelander stated these results for Riemannian foliations on compact manifolds, either by considering the residual set of leaves without holonomy, or by considering the holonomy covers of all leaves. This kind of extension is straightforward by Molino's theory \cite{Molino1988}.

In this paper, we show that the arguments of Breuillard and Gelander about the growth of Lie foliations can be extended to $(G,G/P)$-foliations assuming that $\overline\Gamma$ is connected. Let us remark that a $(G,G/P)$-foliation may not be Riemannian (it is Riemannian if $P$ is compact). Precisely, we prove the following complement of Blumenthal's observations about growth of transversely homogeneous foliations.

\begin{thm}\label{t:growth transv homog folns}
  Let $G$ be a Lie group, and $P\subset G$ a closed subgroup such that $G/P$ is connected and the $G$-action on $G/P$ is faithful. Let $\FF$ be a $(G,G/P)$-foliation on a compact connected manifold with holonomy group $\Gamma\subset G$. If $\overline\Gamma$ is connected, then:
    \begin{itemize}
    
      \item either all holonomy covers of the leaves of $\FF$ have polynomial growth with degree bounded by a common constant;
      
      \item or all holonomy covers of the leaves have exponential growth.
    
    \end{itemize}
\end{thm}

As a particular case, $(\PSL(2;\R),\R P^1)$-foliations are called {\em transversely projective\/}. The first example of a codimension $1$ foliation on a compact $3$-manifold with nonzero Godbillon-Vey invariant, due to Roussarie, was transversely projective, and its holonomy group $\Gamma$ is discrete and uniform in $\PSL(2;\R)$ (see e.g.\ \cite[Example~1.3.14]{CandelConlon2000-I}). Chihi and ben~Ramdane \cite{ChihiRamdane2008} have shown that, for any transversely projective foliation on a compact manifold, if the Godbillon-Vey invariant is nonzero, then $\Gamma$ is either discrete or dense in $\PSL(2;\R)$. Moreover they constructed examples satisfying the second alternative on compact manifolds of dimension $\ge5$, and therefore satisfying the conditions of Theorem~\ref{t:growth transv homog folns}.

We thank the referee for suggestions to improve the paper.

%

\section{Preliminaries}\label{s:prelim}

\subsection{Coarse quasi-isometries and growth of metric spaces}\label{ss:growth metric}

A {\em net\/} in a metric space $M$, with metric $d$, is a subset $A\subset M$ that satisfies $d(x,A)\le C$ for some $C>0$ and all $x\in M$; the term {\em $C$-net\/} is also used.  A {\em coarse quasi-isometry\/} between $M$ and another metric space $M'$ is a bi-Lipschitz bijection between nets of $M$ and $M'$; in this case, $M$ and $M'$ are said to be {\em coarsely quasi-isometric\/} (in the sense of Gromov) \cite{Gromov1993}. If such a bi-Lipschitz bijection, as well as its inverse, has dilation $\le\lambda$, and it is defined between $C$-nets, then it will be said that the coarse quasi-isometry has {\em distortion\/} $(C,\lambda)$. A family of coarse quasi-isometries with a common distortion will be called {\em uniform\/}, and the corresponding metric spaces are called {\em uniformly\/} coarsely quasi-isometric. 

The version of growth for metric spaces given here is taken from \cite{AlvCandel:gcgol}. Since \cite{AlvCandel:gcgol} is not finished yet, some short proofs are included.

Recall that, given non-decreasing functions\footnote{Usually, growth types are defined by using non-decreasing functions $\Z^+\to[0,\infty)$, but this gives rise to an equivalent concept (see \cite{PhillipsSullivan1981}).} $u,v:[0,\infty)\to[0,\infty)$, it is said that $u$ is {\em dominated\/} by $v$, written $u\preccurlyeq v$, when there are $a,b,c,d>0$ such that $u(r)\le a\,v(br+c)+d$ for all $r$. If $u\preccurlyeq v\preccurlyeq u$, then it is said that $u$ and $v$ represent the same {\em growth type\/}; this is an equivalence relation and ``$\preccurlyeq$'' defines a partial order relation between growth types called {\em domination\/}. For a family of pairs of non-decreasing functions $[0,\infty)\to[0,\infty)$, {\em uniform\/} domination means that those pairs satisfy the above condition of domination with the same constants $a,b,c,d$. A family of non-decreasing  functions $[0,\infty)\to[0,\infty)$ will be said to have {\em uniformly\/} the same growth type if they uniformly dominate one another.

For a complete connected Riemannian manifold $L$, the growth type of each mapping $r\mapsto\vol B(x,r)$ is independent of $x$ and is called the {\em growth type\/} of $L$. Another definition of {\em growth type\/} can be similarly given for metric spaces whose bounded sets are finite, where the number of points is used instead of the volume.

Let $M$ be a metric space with metric $d$. A {\em quasi-lattice\/} $\Gamma$ of $M$ is a $C$-net of $M$ for some $C\ge0$ such that, for every $r\ge0$, there is some $K_r\ge0$ such that $\card(\Gamma\cap B(x,r))\le K_r$ for every $x\in M$. It is said that $M$ is of {\em coarse bounded geometry\/} if it has a quasi-lattice. In this case, the {\em growth type\/} of $M$ can be defined as the growth type of any quasi-lattice $\Gamma$ of $M$; i.e., it is the growth type of the {\em growth function\/} $r\mapsto v_\Gamma(x,r)=\card(B(x,r)\cap\Gamma)$ for any $x\in\Gamma$. This definition can be proved to be independent of $\Gamma$ as
follows. Let $\Gamma'$ be another quasi-lattice in $M$. So $\Gamma$ and $\Gamma'$ are $C$-nets in $M$ for some $C\ge0$, and there is some $K_r\ge0$ for each $r\ge0$ such that $\card(B(x,r)\cap\Gamma)\le K_r$ and $\card(B(x,r)\cap\Gamma')\le K_r$ for all $x\in M$. Fix points $x\in\Gamma$ and $x'\in\Gamma'$, and let  $\delta=d(x,x')$. Because $B(x,r)\subset B(x',r+\delta)$ and $\Gamma'$ is a $C$-net, it follows that
$$
B(x,r)\cap\Gamma\subset\bigcup_{y'\in B(x',r+\delta+C)\cap\Gamma'}B(y',C)\cap\Gamma'\;,
$$
yielding
$$
v_\Gamma(x,r)\le K_C\,v_{\Gamma'}(x',r+\delta+C)\le K_C\,v_{\Gamma'}(x',(1+\delta+C)r)
$$
for all $r\ge1$. Hence the growth type of $r\mapsto v_\Gamma(x,r)$ is dominated by the growth type of $r\mapsto v_{\Gamma'}(x',r)$. 

A family of metric spaces which satisfy the above condition of coarse bounded geometry with the same constants $C$ and $K_r$ is said to have {\em uniformly\/} coarse bounded geometry. If moreover the lattices involved in this condition have growth functions defining uniformly the same growth type, then these metric spaces are said to have {\em uniformly\/} the same growth type.

 The condition of coarse bounded geometry is satisfied by complete connected Riemannian manifolds of bounded geometry, and by discrete metric spaces with a uniform upper bound on the number of points in all balls of each given radius \cite{BlockWeinberger1997}. In those cases, the two given definitions of growth type are equal.

\begin{lem}[{\'Alvarez L\'opez-Candel \cite{AlvCandel2009}}]
\label{l:growth type}
  Two coarsely quasi-isometric metric spaces of coarse bounded
  geometry have the same growth type. Moreover, if a family of metric spaces of coarse bounded geometry is uniformly coarsely quasi-isometric, then it has uniformly the same growth type.
\end{lem}

\begin{proof}
  Let $\phi:A\to A'$ be a coarse quasi-isometry between metric spaces $M$ and $M'$ of coarse bounded geometry. Then $A$ is of coarse bounded geometry too, and thus it has some lattice $\Gamma$, which is also a lattice in $M$ because $A$ is a net. Since $\phi$ is a bi-lipschitz bijection, it easily follows that $\Gamma$ and $\phi(\Gamma)$ have the same growth type, and that $\phi(\Gamma)$ is a lattice in $A'$, and thus in $M'$ too because $A'$ is a net. This argument has an obvious uniform version for a family of metric spaces.
\end{proof}

\subsection{Pseudogroups}

A {\em pseudogroup\/} ({\em of local transformations\/}) on a topological space $T$ is a collection $\HH$ of homeomorphisms between open subsets of $T$ that contains the identity map and is closed under composition (wherever defined), inversion, restriction and combination (union) of maps. Such a pseudogroup $\HH$ is {\em generated\/} by a set $E\subset\HH$ if every element of $\HH$ can be obtained from $E$ by using the above pseudogroup operations; usually, $E$ will be symmetric ($h^{-1}\in E$ if $h\in E$). The {\em orbit\/} of a point $x\in T$ is the set $\HH(x)=\{\,h(x)\mid h\in\HH,\ x\in\dom h\,\}$. Many other basic dynamical concepts can be generalized to pseudogroups because they are natural generalizations of dynamical systems (each group action on a topological space generates a pseudogroup). The {\em restriction\/} of $\HH$ to an open subset $U\subset T$ is  the pseudogroup $\HH|_U=\{\,h\in\HH\mid\dom h\cup\im h\subset U\,\}$. It is said that $\HH$ is {\em quasi-analytic\/} when any $h\in\HH$ is the identity around any $x\in\dom h$ if $h$ is the identity on some open set whose closure contains $x$;  {\em quasi-analytic\/} group actions are similarly defined, which means that they generate quasi-analytic pseudogroups.  Pseudogroups that are equivalent in the following sense should be considered to have the same dynamics.

\begin{defn}[{Haefliger \cite{Haefliger1985,Haefliger1988}}]
  Let $\HH$ and $\HH'$ be pseudogroups on respective topological spaces $T$ and $T'$. An {\em equivalence\/} $\Phi:\HH\to\HH'$ is a maximal collection $\Phi$ of homeomorphisms of open subsets of $T$ to open subsets of $T'$ such that:
    \begin{itemize}

      \item If $\phi\in\Phi$, $h\in\HH$ and $h'\in\HH'$, then $h'\phi h\in\Phi$; and

      \item $\HH$ and $\HH'$ are generated by the maps of the form $\psi^{-1}\phi$ and $\psi\phi^{-1}$, respectively, with $\phi,\psi\in\Phi$.

    \end{itemize}
  In this case, $\HH$ and $\HH'$ are said to be {\em equivalent\/}.
\end{defn}
 
If $\Phi:\HH\to\HH'$ is an equivalence, then $\Phi^{-1}=\{\phi^{-1}\ |\ \phi\in\Phi\}$ is an equivalence $\HH'\to\HH$, called the {\em inverse\/} of $\Phi$. An equivalence $\Phi:\HH\to\HH'$ is {\em generated\/} by a subset $\Phi_0\subset\Phi$ if all of the elements of $\Phi$ can be obtained by restriction and combination of composites $h'\phi h$ with $h\in\HH$, $\phi\in\Phi_0$ and $h'\in\HH'$. If $\Phi:\HH\to\HH'$ and $\Psi:\HH'\to\HH''$ are equivalences, then the maps $\psi\phi$, for $\phi\in\Phi$ and $\psi\in\Psi$, generate an equivalence $\Psi\Phi:\HH\to\HH''$, called the {\em composite\/} of $\Phi$ and $\Psi$. Thus the equivalence of pseudogroups is an equivalence relation. An equivalence $\Phi:\HH\to\HH'$ induces a homeomorphism between the corresponding orbit spaces, $\bar\Phi:T/\HH\to T'/\HH'$.

A basic example of a pseudogroup equivalence is the following. Let $\HH$ be a pseudogroup of local transformations of a space $T$, let $U\subset T$ be an open subset that meets every $\HH$-orbit. Then the inclusion map $U\hookrightarrow T$ generates an equivalence
$\HH|_U\to\HH$. In fact, this example can be used to describe any pseudogroup
equivalence in the following way. Let $\HH$ and $\HH'$ be pseudogroups
of local transformations of respective spaces $T$ and $T'$, and let $\Phi:\HH\to\HH'$ be an
equivalence. Let $\HH''$ be the pseudogroup of local transformations of $T''=T\sqcup T'$ generated
by $\HH\cup\HH'\cup\Phi$. Then the inclusions of $T$ and $T'$ in $T''$ generate
equivalences $\Psi_1:\HH\to\HH''$ and $\Psi_2:\HH'\to\HH''$ so that
$\Phi=\Psi_2^{-1}\Psi_1$.

For a pseudogroup $\HH$ on a locally compact space $T$, the orbit space $T/\HH$ is compact if and only if there exists a relatively compact open subset of $T$ that meets every $\HH$-orbit. The following is a stronger ``compactness'' condition on a pseudogroup.

\begin{defn}[{Haefliger \cite{Haefliger1985}}]\label{d:compactly generated}
Let $\HH$ be a pseudogroup on a locally compact space $T$.  It is said that $\HH$ is {\em compactly generated\/} if there is a relatively compact open set $U$ in $T$ meeting each $\HH$-orbit, and such that $\HH|_U$ is generated by a finite symmetric collection $E$ so that each $g\in E$ has an extension $\bar g\in\HH$ with $\overline{\dom g}\subset\dom\bar g$.
\end{defn}

It was observed in \cite{Haefliger1985} that the property of being compactly generated depends only on the equivalence class of the pseudogroup, and that the relatively compact open set $U$ meeting each orbit can be chosen arbitrarily. If $E$ satisfies the conditions of Definition~\ref{d:compactly generated}, it will be called a {\em system of compact generation\/} of $\HH$ on $U$. 

For $h\in\HH$ and $x\in\dom h$, let $\germ(h,x)$ denote the germ of $h$ at $x$. Then
  \[
    \fH=\{\,\germ(h,x)\mid h\in\HH,\ x\in\dom h\,\}
  \]
becomes a topological groupoid, with object space $T$, equipped with the \'etale topology, the operation induced by composition, and the source and target projections to $T$. For $x,y\in T$, let $\fH_x$ (respectively, $\fH^y$) denote the set of elements in $\fH$ with source $x$ (respectively, with target $y$), and let $\fH_x^y=\fH_x\cap\fH^y$; in particular, the group $\fH_x^x$ will be called the {\em germ group\/} of $\HH$ at $x$. Points in the same $\HH$-orbit have isomorphic germ groups (if $y\in\HH(x)$, an isomorphism $\fH_y^y\to\fH_x^x$ is given by conjugation with any element in $\fH_x^y$); hence the germ groups of the orbits make sense up to isomorphism. Under pseudogroup equivalences, corresponding orbits have isomorphic germ groups. The set $\fH_x$ will be called the {\em germ cover\/} of the orbit $\HH(x)$ with base point $x$. The target map restricts to a surjective map $\fH_x\to\HH(x)$ whose fibers are bijective to $\fH_x^x$ (if $y\in\HH(x)$, a bijection $\fH_x^x\to\fH_x^y$ is given by left product with any element in $\fH_x^y$); thus $\fH_x$ is finite if and only if both $\fH_x^x$ and $\HH(x)$ are finite. Moreover germ covers based on points in the same orbit have the same cardinality (if $y\in\HH(x)$, a bijection $\fH_y\to\fH_x$ is given by right product with any element in $\fH_x^y$); therefore the germ covers of the orbits make sense up to bijections.

\subsection{Quasi-isometry type of orbits}\label{ss:q-i}

Let $\HH$ be a pseudogroup on a space $T$, and $E$ a symmetric set of generators of $\HH$. For each $h\in\HH$ and $x\in\dom h$, let $|h|_{E,x}$ be the length of the shortest expression of $\germ(h,x)$ as a product of germs of maps in $E$ (being $0$ if $\germ(h,x)=\germ(\id_T,x)$). For each $x\in T$, define metrics $d_E$ on $\HH(x)$ and $\fH_x$ by
\begin{gather*} 
  d_E(y,z)=\min\{\,|h|_{E,y}\mid h\in\HH,\ y\in\dom h,\ h(y)=z\,\}\;,\\ 
  d_E(\germ(f,x),\germ(g,x))=|fg^{-1}|_{E,g(x)}\;.
\end{gather*}  
Notice that
  \[
    d_E(f(x),g(x))\le d_E(\germ(f,x),\germ(g,x))\;.
  \]
Moreover, on the germ covers, $d_E$ is right invariant in the sense that, if $y\in\HH(x)$, the bijection $\fH_y\to\fH_x$, given by right multiplication with any element in $\fH_x^y$, is isometric; so the isometry types of the germ covers of the orbits make sense without any reference to base points. In fact, the definition of $d_E$ on $\fH_x$ is analogous to the right invariant metric $d_S$ on a group $\Gamma$ defined by a symmetric system of generators $S$: $d_S(\gamma,\delta)=|\gamma\delta^{-1}|$ for $\gamma,\delta\in\Gamma$, where $|\gamma|$ is the length of the shortest expression of $\gamma$ as product of elements of $S$ (being $0$ if $\gamma=1$). 

Assume that $\HH$ is compactly generated and $T$ locally compact. Let $U\subset T$ be a relatively compact open subset that meets all $\HH$-orbits, let $\GG=\HH|_U$, and let $E$ be a symmetric system of compact generation of $\HH$ on $U$. With these conditions, the quasi-isometry type of the $\GG$-orbits with $d_E$ may depend on $E$ \cite[Section~6]{AlvCandel2009}. So the following additional condition on $E$ is considered.

\begin{defn}[{\'Alvarez L\'opez-Candel \cite[Definition~4.2]{AlvCandel2009}}]\label{d:recurrent finite symmetric family of generators} With the above notation, it is said that $E$ is {\em recurrent\/} if, for any relatively compact open subset $V\subset U$ that meets all $\GG$-orbits, there exists some $R>0$ such that $\GG(x)\cap V$ is an $R$-net in $\GG(x)$ with $d_E$ for all $x\in U$.
\end{defn}

Actually, if some relatively compact open subset $V$ of $U$ which meets each orbit satisfies the above condition, then it is satisfied by all such subsets of $U$ \cite[Lemma~4.3]{AlvCandel2009}. Furthermore there always exists a recurrent system of compact generation on $U$ \cite[Corollary~4.5]{AlvCandel2009}.

In other words, the recurrence of $E$ means that there is some $N\in\N$ such that 
     \begin{equation}\label{E^N}
       U=\bigcup_{h\in E^N}h^{-1}(V\cap\im h)\;,
     \end{equation}
where $E^N$ is the family of compositions of at most $N$ elements of $E$.

\begin{thm}[{\'Alvarez L\'opez-Candel \cite[Theorem~4.6]{AlvCandel2009}}]
\label{t:quasi-isometric orbits}
Let $\HH$ and $\HH'$ be compactly generated pseudogroups on locally compact spaces $T$ and $T'$, let $U$ and $U'$ be relatively compact open subsets of $T$ and $T'$ that meet all orbits of $\HH$ and $\HH'$, let $\GG$ and $\GG'$ denote the restrictions of $\HH$ and $\HH'$ to $U$ and $U'$, and let $E$ and $E'$ be recurrent symmetric systems of compact generation of $\HH$ and $\HH'$ on $U$ and $U'$, respectively. Suppose that there exists an equivalence $\HH\to\HH'$, and consider the induced equivalence $\GG\to\GG'$ and homeomorphism $U/\GG\to U'/\GG'$. Then the $\GG$-orbits with $d_E$ are uniformly quasi-isometric to the corresponding $\GG'$-orbits with $d_{E'}$.
\end{thm}

An obvious modification of the arguments of the proof of \cite[Theorem~4.6]{AlvCandel2009} gives the following.

\begin{thm}\label{t:quasi-isometric germ covers of orbits}
  With the notation and conditions of Theorem~\ref{t:quasi-isometric orbits}, the germ covers of the $\GG$-orbits with $d_E$ are uniformly quasi-isometric to the germ covers of the corresponding $\GG'$-orbits with $d_{E'}$.
\end{thm}

\begin{cor}\label{c:growth orbits and their germ covers}
  With the notation and conditions of Theorems~\ref{t:quasi-isometric orbits} and~\ref{t:quasi-isometric germ covers of orbits}, the corresponding orbits of $\GG$ and $\GG'$, as well as their germ covers, have the same growth type, uniformly.
\end{cor}
  
\begin{proof}
  This follows from Lemma~\ref{l:growth type} and Theorems~\ref{t:quasi-isometric orbits} and~\ref{t:quasi-isometric germ covers of orbits}.
\end{proof}

\subsection{Growth of leaves}\label{ss:growth leaves}

Let us recall some basic concepts of foliation theory (see e.g.\ \cite{HectorHirsch1981-A,HectorHirsch1983-B,Godbillon1991,CandelConlon2000-I,Walczak2004}). Let $\FF$ be a smooth foliation on a manifold $M$ given by a defining cocycle $(U_i,p_i,h_{ij})$ \cite{Haefliger1985,Haefliger1988}, where $p_i:U_i\to T_i$, and $h_{ij}:p_i(U_i\cap U_j)\to p_j(U_i\cap U_j)$ is determined by the condition $p_j=h_{ij}p_i$ on $U_i\cap U_j$. We can assume that $(U_i,p_i,h_{ij})$ is induced by a regular foliation atlas: the sets $U_i$ are the domains of the foliation charts, the maps $p_i$ are the local projections whose fibers are the plaques, and the maps $h_{ij}$ are the transverse components of the changes of coordinates. The equivalence class of the pseudogroup $\HH$ on $T=\bigsqcup T_i$ generated by the maps $h_{ij}$ is independent of the choice of defining cocycle, and is called the {\em holonomy pseudogroup\/} of $\FF$. There is a canonical identity between the space of leaves and the space of $\HH$-orbits, $M/\FF\equiv T/\HH$. The {\em holonomy group\/} of each leaf $L$ is defined as the germ group of the corresponding orbit. It can be considered as a quotient of the fundamental group of $L$ by taking ``chains'' of sets $U_i$ along loops in $L$, and the corresponding covering space is called the {\em holonomy cover\/} $\widetilde{L}$ of $L$. If $\FF$ admits a countable defining cocycle, then the leaves in some saturated residual subset of $M$ have trivial holonomy groups  \cite[Lemme~1]{Hector1977}, \cite{EpsteinMillettTischler1977}, and therefore they can be identified with their holonomy covers.

Suppose that $M$ is compact. Then, given any Riemannian metric $g$ on $M$, for each leaf $L$, the differentiable (and coarse) quasi-isometry types of $g|_L$ and its lift to $\widetilde{L}$ are independent of the choice of $g$; they depend only on $\FF$ and $L$. On the other hand, $\HH$ is compactly generated \cite{Haefliger1985}, which can be seen as follows. There is some defining cocycle $(U'_i,p'_i,h'_{ij})$, with $p'_i:U'_i\to T'_i$, such that $\overline{U_i}\subset U'_i$, $\overline{T_i}\subset T'_i$ and $p'_i$ extends $p_i$. Therefore each $h'_{ij}$ is an extension of $h_{ij}$ so that $\overline{\dom h_{ij}}\subset\dom h'_{ij}$. Moreover $\HH$ is the restriction to $T$ of the pseudogroup $\HH'$ on $T'=\bigsqcup_iT'_i$ generated by the maps $h'_{ij}$, and $T$ is a relatively compact open subset of $T'$ that meets all $\HH'$-orbits.

The collection $E$ of maps $h_{ij}$ is a recurrent system of compact generation of $\HH'$ on $T$ \cite[Lemma~5.4]{AlvCandel:gcgol}. According to Theorems~\ref{t:quasi-isometric orbits} and~\ref{t:quasi-isometric germ covers of orbits}, it follows that the quasi-isometry type of the $\HH$-orbits and their germ covers with $d_E$ are independent of the choice of $(U_i,p_i,h_{ij})$ under the above conditions; in fact, they are coarsely quasi-isometric to the corresponding leaves, and therefore they have the same growth type \cite{Carriere1988} (this is an easy consequence of the existence of a uniform lower bound and upper bound of the diameter and volume of the plaques). Similarly, the germ covers of the $\HH$-orbits are also quasi-isometric to the holonomy covers of the corresponding leaves.

\section{Growth of homogeneous pseudogroups}\label{s:growth homog pseudogroups}

Let $G$ be a Lie group, and $P\subset G$ a closed subgroup such that the $G$-action on $G/P$ is faithful and that $G/P$ is connected; thus this action is also quasi-analytic by the analyticity of $G/P$. Let us define a {\em {\rm(}homogeneous{\rm)} $(G,G/P)$-pseudogroup\/} as a pseudogroup equivalent to the pseudogroup $\HH$ generated by the action of some subgroup $\Gamma\subset G$ on some $\Gamma$-invariant open subset $T\subset G/P$. Suppose that $\HH$ is compactly generated, and let $\GG=\HH|_U$ for some relatively compact open subset $U\subset T$ that meets all $\HH$-orbits. For every $\gamma\in\Gamma$ with $\gamma\cdot U\cap U\ne\emptyset$, let $h_\gamma$ denote the restriction $U\cap\gamma^{-1}\cdot U\to\gamma\cdot U\cap U$ of the left translation $\gamma\cdot:G\to G$. There is a finite symmetric set $S=\{s_1,\dots,s_k\}\subset\Gamma$ such that $E=\{h_{s_1},\dots,h_{s_k}\}$ is a recurrent system of compact generation of $\HH$ on $U$; in fact, by reducing $\Gamma$ if necessary, we can assume that $S$ generates $\Gamma$. For each $x\in U$, let
  \[
    \Gamma_{U,x}=\{\,\gamma\in\Gamma\mid \gamma\cdot x\in U\,\}\;.
  \]
Let $\fG$ denote the topological groupoid of germs of $\GG$. Since the $G$-action on $G/P$ is faithful and quasi-analytic, and $G/P$ is connected, we get a bijection $\Gamma_{U,x}\to\fG_x$, $\gamma\mapsto\germ(h_\gamma,x)$. For $\gamma\in\Gamma_{U,x}$, let $|\gamma|_{S,U,x}:=|h_\gamma|_{E,x}$. Thus $|1|_{S,U,x}=0$, and, if $\gamma\ne1$, then $|\gamma|_{S,U,x}$ equals the minimum $n\in\N$ such that there are $i_1,\dots,i_n\in\{1,\dots,k\}$ with  $\gamma=s_{i_n}\cdots s_{i_1}$ and $s_{i_m}\cdots s_{i_1}\cdot x\in U$ for all $m\in\{1,\dots,n\}$. Moreover $d_E$ on $\fG_x$ corresponds to the metric $d_{S,U,x}$ on $\Gamma_{U,x}$ given by
  \[
    d_{S,U,x}(\gamma,\delta)=|\delta\gamma^{-1}|_{S,U,\gamma(x)}\;.
  \]

\begin{thm}\label{t:growth homog pseudogroup}
  With the above notation and conditions, if $\overline\Gamma$ is connected, then
    \begin{itemize}
    
      \item either all germ covers of the $\GG$-orbits have polynomial growth with degree bounded by some common constant;
      
      \item or all infinite germ covers of the $\GG$-orbits have exponential growth.
    
    \end{itemize}
\end{thm}

\begin{proof}
  We can suppose that $\Gamma_{U,x}$ is infinite for all $x\in U$, otherwise $\Gamma=\{1\}$ because $\overline\Gamma$ is connected.
  
  Fix any open set $V$ that meets all orbits such that $\overline V\subset U$. By~\eqref{E^N}, after increasing $S$ if necessary, we can suppose that
 	\begin{equation}\label{U subset bigcup_i s_i cdot V}
      		U\subset\bigcup_is_i\cdot V\;.
    	\end{equation}
  Actually, given any relatively compact open subset $U'\subset T$ such that $\overline{U}\subset U'$, we can increase $S$ again to produce a recurrent compact system of generators of $\HH$ on $U'$. Then~\eqref{U subset bigcup_i s_i cdot V} can be applied with $U'$ instead of $U$, obtaining
  	\begin{equation}\label{overline U subset bigcup_i s_i cdot V}
      		\overline U\subset\bigcup_is_i\cdot V\;.
    	\end{equation}
  Since $\overline U$ is compact and $\Gamma$ is not discrete, we can choose elements $s'_1,\dots,s'_k\in\Gamma\setminus S$ close enough to $s_1,\dots,s_k$ such that~\eqref{overline U subset bigcup_i s_i cdot V} also holds using the elements $s'_1,\dots,s'_k$ instead of the elements $s_1,\dots,s_k$. Hence, by adding the elements $s'_1,\dots,s'_k$ and their inverses to $S$, we can also suppose 
    \begin{equation}\label{s_i}
      \overline{U}\subset\bigcup_{i<j}(s_i\cdot V\cap s_j\cdot V)=\bigcup_{i<j}(s_i^{-1}\cdot V\cap s_j^{-1}\cdot V)\;.
    \end{equation}
  
   Assume first that $\Gamma$ is not nilpotent. Thus $\overline\Gamma$ is a non-nilpotent connected Lie group, and therefore we can apply \cite[Proposition~10.5]{BreuillardGelander2007} to its finitely generated dense subgroup $\Gamma$, obtaining that there are elements $t_1,\dots,t_k$ in $\Gamma$, as close as desired to $s_1,\dots,s_k$, respectively, which are free generators of a free semi-group. By the compactness of $\overline U$, if $t_1,\dots,t_k$ are close enough to $s_1,\dots,s_k$, then~\eqref{s_i} gives 
    \begin{equation}\label{t_i}
      \overline U\subset\bigcup_{i<j}(t_i^{-1}\cdot V\cap t_j^{-1}\cdot V)\;.
    \end{equation}
  
  Now, we adapt the argument of the proof of \cite[Lemma~10.6]{BreuillardGelander2007}. Let $\Gamma'\subset\Gamma$ be the subgroup generated by $t_1,\dots,t_k$; thus $S'=\{t_1^{\pm1},\dots,t_k^{\pm1}\}$ is a symmetric set of generators of $\Gamma'$, and $S\cup S'$ is a symmetric set of generators of $\Gamma$. With $E'=\{h_{t_1}^{\pm1},\dots,h_{t_k}^{\pm1}\}$, observe that $E\cup E'$ is a recurrent system of compact generation of $\HH$ on $U$. Given $x\in U$, let $\mathsf{S}(n)$ be the sphere with center the identity element and radius $n\in\N$ in $\Gamma'_{U,x}$ with $d_{S',U,x}$. Let us construct a subset $\mathsf{T}(n)\subset\mathsf{N}(n)$ by induction on $n$. Let $\mathsf{T}(0)=\mathsf{N}(0)$. Now, assume that $\mathsf{T}(n)$ is defined. By~\eqref{t_i}, for each $\gamma\in\mathsf{T}(n)$, we have $\gamma\cdot x\in t_i^{-1}\cdot V\cap t_j^{-1}\cdot V$ for some indices $i<j$. So the points $t_i\gamma\cdot x$ and $t_j\gamma\cdot x$ are in $V$, obtaining that $t_i\gamma,t_j\gamma\in\mathsf{S}(n+1)$. Let $\mathsf{T}(n+1)$ be the set of all elements obtained in this way from elements of $\mathsf{T}(n)$, which are pairwise distinct because $t_1,\dots,t_k$ freely generate a free semigroup. Hence $\card(\mathsf{T}(n+1))\ge2\card(\mathsf{T}(n))$, giving $\card(\mathsf{S}(n))\ge\card(\mathsf{T}(n))\ge2^n$. So $\Gamma'_{U,x}$ has exponential growth with $d_{S',U,x}$. Since $\Gamma'_{U,x}\subset\Gamma_{U,x}$ and $d_{S\cup S',U,x}\le d_{S',U,x}$ on $\Gamma'_{U,x}$, it follows that $\Gamma_{U,x}$ also has exponential growth with $d_{S\cup S',U,x}$. So $\fG_x$ has exponential growth with $d_{E\cup E'}$, obtaining that $\fG_x$ has exponential growth with $d_E$ by Corollary~\ref{c:growth orbits and their germ covers}.
    
  If $\Gamma$ is nilpotent, then it has polynomial growth with respect to $d_S$, and this growth type dominates the growth type of $\Gamma_{U,x}$ with the metric $d_{S,U,x}$ because $d_S\le d_{S,U,x}$ on $\Gamma_{U,x}$.
\end{proof}

Now, according to \cite[Theorem~1]{Blumenthal1979} (see Section~\ref{s: intro}), Theorem~\ref{t:growth transv homog folns} is a direct consequence of Theorem~\ref{t:growth homog pseudogroup} and the observations of Section~\ref{ss:growth leaves}.



\providecommand{\bysame}{\leavevmode\hbox to3em{\hrulefill}\thinspace}
\providecommand{\MR}{\relax\ifhmode\unskip\space\fi MR }
\providecommand{\MRhref}[2]{%
  \href{http://www.ams.org/mathscinet-getitem?mr=#1}{#2}
}
\providecommand{\href}[2]{#2}

\end{document}